\documentclass[12pt]{article}
\usepackage[english]{babel}
\usepackage{geometry,amsmath,amsthm,amssymb,latexsym}

\usepackage[colorlinks=true,linkcolor=blue,citecolor=blue,urlcolor=black,pdfpagelabels=false]{hyperref}

\usepackage{thmtools}
\theoremstyle{plain}
  \declaretheorem[numberwithin=section]{theorem}

  \declaretheorem[numberlike=theorem]{lemma}

\theoremstyle{definition}
  
  \declaretheorem[numberlike=theorem]{example}
  \declaretheorem[numberlike=theorem]{remark}

\newcommand{\assign}{:=}

\begin{document}

\title
{Sums of powers of binomials, their Ap\'ery limits, and~Franel's~suspicions}

\author{
  Armin Straub
  \and
  Wadim Zudilin
}

\date{March 19, 2022}

\maketitle

\begin{abstract}
  We explicitly determine the Ap\'ery limits for the sums of powers of
  binomial coefficients. As an application, we prove a weak version of
  Franel's conjecture on the order of the recurrences for these sequences.
  Namely, we prove the conjectured minimal order under the assumption that
  such a recurrence can be obtained via creative telescoping.
\end{abstract}

\section{Introduction}

More than a century ago, Franel \cite{franel94,franel95} investigated the
sums of integral powers of binomial coefficients
\begin{equation}
  A^{(s)} (n) = \sum_{k = 0}^n \binom{n}{k}^s . \label{eq:franel:d}
\end{equation}
The special cases $A^{(1)} (n) = 2^n$ and $A^{(2)} (n) = \binom{2 n}{n}$ are
simple. On the other hand, the numbers $A^{(3)} (n)$, known as \emph{Franel
numbers} \cite[A000172]{oeis}, cannot be expressed as a finite linear
combination of hypergeometric terms \cite[p.~160]{aeqb}. We will refer to
the numbers $A^{(s)} (n)$ as the generalized Franel numbers. Long before the
computer-algebra era, Franel \cite{franel94} computed recurrences for
$A^{(3)} (n)$ as well as, in the second note \cite{franel95}, for $A^{(4)}
(n)$. Based on these findings, he predicted\,---\,quite optimistically\,---\,a
general shape of the recursion for general $s$. Since then, explicit
recurrences for $A^{(s)} (n)$ have been computed using creative telescoping by
Perlstadt \cite{perlstadt-franel} for $s = 5, 6$ and, likewise, by McIntosh
\cite{mcintosh-phd} for $s \leq 10$. Creative telescoping, which we
briefly review in Section~\ref{sec:ct}, is a powerful computer-algebra
technique that can, for fixed integer $s$, algorithmically determine a
recurrence satisfied by $A^{(s)} (n)$. More specifically, given a
hypergeometric term like $a (n, k) = \binom{n}{k}^s$, it produces an operator
$P (n, N)$ (here, $N$ is the shift operator in $n$: $N a (n, k) \assign a (n +
1, k)$), as well as another hypergeometric term $b (n, k)$, such that
\begin{equation}
  P (n, N) a (n, k) = b (n, k + 1) - b (n, k) . \label{eq:ct:Pb:intro}
\end{equation}
Summing the relation \eqref{eq:ct:Pb:intro} (with some care and under some
mild assumptions; see the beginning of Section~\ref{sec:ct}) over all integers
$k$, the contribution of $b (n, k)$ telescopes away, allowing us to conclude
that $A^{(s)} (n)$ is annihilated by the operator $P (n, N)$; in this case, we
say that $A^{(s)} (n)$ satisfies the \emph{telescoping recurrence equation}
$P (n, N) A^{(s)} (n) = 0$. Notice that this telescoping equation is based on
the representation \eqref{eq:franel:d}, as it uses the operator $P (n, N)$ for
the hypergeometric term $a (n, k) = \binom{n}{k}^s$. Using a different
hypergeometric representation\,---\,and such exist (for example, $A^{(3)} (n)
= \sum_{k = 0}^n \binom{n}{k}^2 \binom{2 k}{n}$)\,---\,may potentially lead to
a different operator.

Franel's suspicions about the form of linear recurrence with polynomial
coefficients for $A^{(s)} (n)$ are not supported by computations in
\cite{mcintosh-phd,perlstadt-franel} with the exception of one particular
aspect, its order. Specifically, Franel conjectured it to be \emph{equal} to
$\lfloor (s + 1) / 2 \rfloor$. While the fact that the order of the recurrence is
bounded from above by this quantity is shown to be true by Stoll
\cite{stoll-rec-bounds}, who indicates that the earlier proof of Cusick
\cite{cusick-rec} has a gap, it remains open to demonstrate that, in
general, no recurrence of lower order exists. The possibility for $A^{(s)}
(n)$ for $s \geq 3$ to satisfy a recurrence of order~$1$, equivalently,
to be a hypergeometric term in the single variable $n$, can be ruled out using
the algorithm \texttt{Hyper} \cite{aeqb}, when $s$ is fixed. On the other
hand, using congruential properties, Yuan, Lu and Schmidt \cite{yls-franel}
prove that, for \emph{any} $s \geq 3$, the sequence $A^{(s)} (n)$
cannot satisfy a recurrence of order~$1$. This implies that Franel's
recurrences of order $2$ for $s = 3$ and $s = 4$ are of minimal order. In
general, to prove that the order $\lfloor (s + 1) / 2 \rfloor$ recurrence
constructed in \cite{stoll-rec-bounds} for the sequence $A^{(s)} (n)$
is of minimal order, it is sufficient to show that the corresponding
recurrence operator is irreducible (though this is not a necessary condition).
For fixed (and sufficiently small) $s$, the latter task is again
accessible for modern computer-algebra algorithms
\cite{bronstein-factoring,hz-fast} (for an explicit example of bounding the
possible degree of a lower-order recurrence, we also refer to the proof of
Proposition~8.4 in \cite[pp.~692--694]{bbkm-latticewalks}). One goal of this
paper is to address the problem for \emph{generic} $s$ by showing the
following general result.

\begin{theorem}
  \label{thm:franel:conj:ct}Any telescoping recurrence satisfied by $A^{(s)}
  (n)$ based on the representation \eqref{eq:franel:d} has order at least
  $\lfloor (s + 1) / 2 \rfloor$.
\end{theorem}

In particular\,---\,in light of \cite{stoll-rec-bounds}\,---\,this implies that
Franel's conjecture on the exact order is true if the minimal-order recurrence
satisfied by $A^{(s)} (n)$ is a telescoping recurrence equation. We refer to
Remark~\ref{rk:franel:recs} for evidence that this is the case.

\begin{remark}
  \label{rk:mcintosh}One way of establishing lower bounds on the order of
  recurrences satisfied by a $D$-finite sequence $A (n)$ comes from the
  observation by McIntosh \cite[Section~4.1,~p.~27]{mcintosh-phd} that, if
  the sequence has the property that $A (n + 1) / A (n) \rightarrow \mu$ where
  $\mu$ is an algebraic number of degree $d$, then $A (n)$ cannot satisfy a
  recurrence defined over $\mathbb{Q}$ of order less than $d$. For the
  generalized Franel numbers $A^{(s)} (n)$, however, it follows from
  \eqref{eq:franel:asy} that $A^{(s)} (n + 1) / A^{(s)} (n) \rightarrow 2^s$,
  so that this criterion is of no help.
\end{remark}

Ap\'ery's groundbreaking proof \cite{apery,alf} of the irrationality of $\zeta
(3)$ is centred around the fact that
\begin{equation}
  \lim_{n \rightarrow \infty} \frac{B (n)}{A (n)} = \frac{\zeta (3)}{6},
  \label{eq:apery3:lim}
\end{equation}
where the sequences
\begin{equation}
  A (n) = \sum_{k = 0}^n \binom{n}{k}^2 \binom{n + k}{k}^2 \label{eq:apery3}
\end{equation}
and $B (n)$ both are solutions to the three-term recurrence
\begin{equation}
  (n + 1)^3 u_{n + 1} = (2 n + 1) (17 n^2 + 17 n + 5) u_n - n^3 u_{n - 1}
  \label{eq:apery3:rec}
\end{equation}
with initial conditions $A (0) = 1$, $A (1) = 5$ as well as $B (0) = 0$ and $B
(1) = 1$. Limits, like \eqref{eq:apery3:lim}, of quotients of solutions to a
common linear recurrence are refered to as \emph{Ap\'ery limits}. For an
introduction to such limits we refer to \cite{cs-aperylimits} as well as to
the papers \cite{avz-apery-limits,yang-apery-limits}. The main
goal of this paper is to explicitly determine the Ap\'ery limits associated
to the generalized Franel numbers $A^{(s)} (n)$. In fact, we will then prove
Theorem~\ref{thm:franel:conj:ct}, discussed above, in
Section~\ref{sec:ct:order} as an application of these Ap\'ery limits.

It was conjectured in \cite{cs-aperylimits} that, for $s \geq 2 m + 1$,
the minimal-order recurrence satisfied by $A^{(s)} (n)$ has Ap\'ery limits
that are rational multiples of $\zeta (2), \zeta (4), \ldots, \zeta (2 m)$.
More precisely, this means that the recurrence has rational solutions
$A^{(s)}_j (n)$, where $j \in \{ 0, 1, \ldots, m \}$, (with $A^{(s)}_0 (n) = A^{(s)}
(n)$) such that
\begin{equation*}
  \lim_{n \rightarrow \infty} \frac{A^{(s)}_j (n)}{A^{(s)} (n)} \in \pi^{2 j}
   \mathbb{Q}.
\end{equation*}
In Theorem~\ref{thm:franel:limits}, we prove this conjecture, with the
minimal-order recurrence replaced by the minimal-order telescoping recurrence,
and explicitly describe all of these Ap\'ery limits. In particular, in terms
of
\begin{equation}
  A^{(s)} (n, t) \assign \sum_{k = 0}^n \binom{n}{k}^s \left[ \prod_{j = 1}^k
  \left(1 - \frac{t}{j} \right) \prod_{j = 1}^{n - k} \left(1 + \frac{t}{j}
  \right) \right]^{- s} =A^{(s)} (n, -t), \label{eq:A:t}
\end{equation}
we identify specific solutions $A^{(s)}_j (n) \in\mathbb{Q}$ as the coefficients
in the $t$-expansion
\begin{equation}
  A^{(s)} (n, t) = \sum_{j \geq 0} A^{(s)}_j (n) t^{2 j} . \label{eq:A:j}
\end{equation}

\begin{theorem}
  \label{thm:franel:limits}Any telescoping recurrence satisfied by $A^{(s)}(n)$
based on the representation \eqref{eq:franel:d}
  is solved, for large enough $n$, by the sequences $A^{(s)}_j(n) \in\mathbb{Q}$
  defined in \eqref{eq:A:j}, where $j \in \{ 0, 1, \ldots, \lfloor (s -
  1) / 2 \rfloor \}$. Furthermore, we have
  \begin{equation}
    \lim_{n \rightarrow \infty} \frac{A^{(s)}_j (n)}{A^{(s)} (n)} =
    \varphi_j^{(s)} \pi^{2 j}, \label{eq:franel:limits}
  \end{equation}
  where $\varphi_j^{(s)} \in \mathbb{Q}_{> 0}$ is the coefficient of $t^{2 j}$
  in the power series of $(t / \sin (t))^s$.
\end{theorem}

\begin{proof}
  First, we show in Theorem~\ref{thm:A:t:rec} that, for large enough $n$,
  $A^{(s)} (n, t)$ satisfies the telescoping recurrence up to terms that are
  $O (t^s)$. Second, we prove in Theorem~\ref{thm:franel:lim:sin} that
  \begin{equation}
    \lim_{n \rightarrow \infty} \frac{A^{(s)} (n, t)}{A^{(s)} (n)} = \left(\frac{\pi t}{\sin (\pi t)} \right)^s, \label{eq:franel:lim:sin:intro}
  \end{equation}
  and that the convergence is locally uniform in $t$ (restricted to the unit
  ball $| t | < 1$). Recall that, if analytic functions $f_n$ converge locally
  uniformly to a function $f$, then $f$ is analytic and the derivatives of
  $f_n$ converge to the corresponding derivatives of $f$. Since the terms on
  the left-hand side of $\eqref{eq:franel:lim:sin:intro}$ are analytic in $t$,
  locally uniform convergence allows us to compare the derivatives on both
  sides, so that \eqref{eq:franel:limits} follows. We note that
  \begin{align*}
    \frac{\pi t}{\sin (\pi t)} & = \sum_{j = 1}^{\infty} \left(2 -
    \frac{1}{2^{2 j - 2}} \right) \zeta (2 j) t^{2 j}\\
    & = \sum_{j = 1}^{\infty} \left(\frac{1}{2^{2 j - 1}} - 1 \right)
    \frac{B_{2 j}}{(2 j) !}  (2 \pi i t)^{2 j},
  \end{align*}
  where the expansion in terms of zeta values makes it transparent that
  $\varphi_j^{(s)}$ is positive, while the rationality of the
  $\varphi_j^{(s)}$ is obvious from the series rewritten in terms of Bernoulli
  numbers.
\end{proof}

\begin{example}
  \label{eg:A:init}Note that $A^{(s)} (n, t)$, as defined in \eqref{eq:A:t},
  has the initial values
  \begin{equation*}
    A^{(s)} (0, t) = 1, \quad A^{(s)} (1, t) = \frac{1}{(1 - t)^s} +
     \frac{1}{(1 + t)^s} = 2 \sum_{j \geq 0} \binom{2 j + s - 1}{2 j}
     t^{2 j} .
  \end{equation*}
  Consequently, for $j \geq 1$, the sequences $A^{(s)}_j (n)$ have the
  initial values $A^{(s)}_j (0) = 0$ and $A^{(s)}_j (1) = 2 \binom{2 j + s -
  1}{2 j}$.
\end{example}

\begin{example}
  \label{eg:franel:limits:zeta2}
  Let us consider the special case $j = 1$ of Theorem~\ref{thm:franel:limits}.
  As noted in Example~\ref{eg:A:init}, we have $A^{(s)}_1 (1) = s (s + 1)$. In
  terms of $B^{(s)} (n) = A^{(s)}_1 (n) / A^{(s)}_1 (1)$, the initial
  values are normalised to $B^{(s)} (0) = 0$ and $B^{(s)} (1) = 1$, and
  the Ap\'ery limit \eqref{eq:franel:limits} takes the form
  \begin{equation}
    \lim_{n \rightarrow \infty} \frac{B^{(s)} (n)}{A^{(s)} (n)} = \frac{1}{s
    (s + 1)} \frac{s}{6} \pi^2 = \frac{\zeta (2)}{s + 1},
    \label{eq:franel:limits:zeta2}
  \end{equation}
  which matches \cite[Conjecture~9]{cs-aperylimits} (we note that this
  conjecture further claims that the sequence $B^{(s)} (n)$ is the unique
  solution of the minimal-order recurrence satisfied by $A^{(s)} (n)$ with the
  above properties). The cases $s = 3$ and $s = 4$ of
  \eqref{eq:franel:limits:zeta2} had been numerically observed by Tom Cusick
  \cite[p.~202]{alf}, while the case $s = 5$ appears as a conjecture in
  \cite[Section~4.1]{avz-apery-limits}. The case $s = 3$ was previously
  proved by Zagier \cite{zagier4} using modular forms.
\end{example}

\begin{remark}
  Dougherty-Bliss and Zeilberger \cite{dbz-apery-limits} explore Ap\'ery
  limits related to those of Example~\ref{eg:franel:limits:zeta2} in a
  different direction. They construct a particular sequence $\tilde{B}^{(s)}
  (n) \in \mathbb{Q}$ such that \eqref{eq:franel:limits:zeta2} holds with
  $B^{(s)} (n)$ replaced by $\tilde{B}^{(s)} (n)$. For fixed $s$, the sequence
  $\tilde{B}^{(s)} (n)$ is $D$-finite, which implies that $A^{(s)} (n)$ and
  $\tilde{B}^{(s)} (n)$ satisfy a common linear recurrence (namely, the
  recurrence obtained from the least common left multiple of the two
  individual recurrence operators), but that recurrence is not minimal unless
  $\tilde{B}^{(s)} (n)$ happens to solve the minimal recurrence satisfied by
  $A^{(s)} (n)$ (that this is not the case is readily checked for small $s$).
  We note that one also obtains the limits \eqref{eq:franel:limits:zeta2} for
  the alternative choice $\tilde{B}^{(s)} (n) = A^{(s)} (n) b (n) / (s + 1)$
  where $b (n)$ is any holonomic sequence such that $b (n) \rightarrow \zeta
  (2)$ as $n \rightarrow \infty$. For instance, one could choose $b (n) =
  \sum_{k = 1}^n \frac{1}{k^2}$ or $b (n) = 3 \sum_{k = 1}^n \frac{1}{k^2}
  \binom{2 k}{k}^{- 1}$, where the latter is due to Ap\'ery \cite{apery,alf}
  and converges at an exponential rate. On the other hand, Dougherty-Bliss and
  Zeilberger \cite{dbz-apery-limits} hope that their construction has
  the potential for better irrationality measures.
\end{remark}

\begin{example}
  Likewise, for the case $j = 2$ of Theorem~\ref{thm:franel:limits}, we have
  $A^{(s)}_2 (1) = s (s + 1) (s + 2) (s + 3) / 12$. If we let $C^{(s)} (n) =
  A^{(s)}_2 (n) / A^{(s)}_2 (1)$, then the initial values are normalised
  to $C^{(s)} (0) = 0$ and $C^{(s)} (1) = 1$, and the Ap\'ery limit
  \eqref{eq:franel:limits} takes the form
  \begin{equation*}
    \lim_{n \rightarrow \infty} \frac{C^{(s)} (n)}{A^{(s)} (n)} = \frac{12}{s
     (s + 1) (s + 2) (s + 3)} \frac{s (5 s + 2)}{360} \pi^4 = \frac{3 (5 s +
     2)}{(s + 1) (s + 2) (s + 3)} \zeta (4),
  \end{equation*}
  which matches \cite[Conjecture~11]{cs-aperylimits}.
\end{example}

\section{Solutions of the telescoping recurrence}\label{sec:ct}

We refer to \cite{aeqb,koutschan-phd,chyzak-abc} for
general introductions to creative telescoping. For our purposes, suppose that
we are interested in a sequence
\begin{equation*}
  A (n) = \sum_{k = \alpha}^{\beta - 1} a (n, k) .
\end{equation*}
If $a (n, k)$ is an appropriate hypergeometric term, then creative telescoping
algorithmically determines operators $P (n, N)$ as well as another
hypergeometric term $b (n, k)$, such that
\begin{equation}
  P (n, N) a (n, k) = b (n, k + 1) - b (n, k) . \label{eq:ct:Pb}
\end{equation}
Moreover, the term $b (n, k)$ as produced by creative telescoping is of the form $b (n,
k) = R (n, k) a (n, k)$ for some rational function $R (n, k)$.
When the hypergeometric term $a(n,k)$ is defined over the ring $\mathbb{Z}$
(and this is specifically the case of our interest here, though the argument below extends to other rings),
that is, when both $a(n+\nobreak1,k)/a(n,k)$ and $a(n,k+1)/a(n,k)$ are quotients of polynomials from $\mathbb{Z}[n,k]$,
we can take $P(n,N) \in \mathbb{Z}[n,N]$ and we have $b(n,k)$ defined over $\mathbb{Z}$.
We note that, given $P (n, N)$ and $R (n, k)$, an
identity like \eqref{eq:ct:Pb} can be verified by dividing both sides by $a
(n, k)$, upon which one obtains an identity between rational functions. For
that reason, $R (n, k)$ is refered to as the \emph{certificate} of the
telescoping relation \eqref{eq:ct:Pb}.

It follows from the telescoping nature of \eqref{eq:ct:Pb} that, after summing
over $k$,
\begin{equation}
  P (n, N) \sum_{k = \alpha}^{\beta - 1} a (n, k) = b (n, \beta) - b (n,
  \alpha), \label{eq:ct:Pb:sum}
\end{equation}
assuming that $b (n, k)$ is finite for the involved values of $n$ and $k$.

For our present purposes, $a (n, k) = \binom{n}{k}^s$. We say that $P (n, N)$
is a \emph{telescoping recurrence operator} for the generalized Franel
numbers $A^{(s)} (n)$ if
\begin{equation}
  P (n, N) \binom{n}{k}^s = b (n, k + 1) - b (n, k), \label{eq:franel:ct}
\end{equation}
where $b (n, k) / a (n, k) = R (n, k)$ is a rational function.
We next show that it follows from \eqref{eq:franel:ct} not only that $P (n, N)
A^{(s)} (n) = 0$ but that the same recurrence is also solved by $A^{(s)} (n,
t)$, as defined in \eqref{eq:A:t}, up to terms of order $t^s$ or higher.
Equivalently, the sequences $A^{(s)}_j (n) \in \mathbb{Q}$, as in
\eqref{eq:A:j}, are solutions for $j \in \{ 0, 1, \ldots, \lfloor (s - 1) / 2
\rfloor \}$.

\begin{theorem}
  \label{thm:A:t:rec}For fixed $s$, suppose that $P (n, N)$ is
a telescoping recurrence operator for the generalized Franel
numbers $A^{(s)} (n)$. Then, for large enough $n$, as $t \to 0$,
  \begin{equation}
    P (n, N) A^{(s)} (n, t) = O (t^s) . \label{eq:A:t:rec}
  \end{equation}
\end{theorem}

\begin{proof}
  Using the reflection formula
  \begin{equation*}
    \Gamma (t) \Gamma (1 - t) = \frac{\pi}{\sin (\pi t)},
  \end{equation*}
  we find that, for $n \in \mathbb{Z}_{\geq 0}$,
  \begin{align}
    \binom{n}{- t} &= \frac{\Gamma (n + 1)}{\Gamma (n + t + 1) \Gamma (1 - t)}
     = \frac{\sin (\pi t)}{\pi} \frac{\Gamma (n + 1) \Gamma (t)}{\Gamma (n + t + 1)} \nonumber
     \\&= \frac{\sin (\pi t)}{\pi}  \frac{n!}{t (t + 1) \cdots (t + n)} .
     \label{eq:binom:neg}
  \end{align}
  Consequently, for $k \in \mathbb{Z}$ such that $0 \leq k \leq n$,
  \begin{align}
    \binom{n}{k - t} & = \frac{\sin (\pi t)}{\pi t}  \frac{(- 1)^k n!}{(t -
    k) \cdots (t - 1) (t + 1) \cdots (t + n - k)} \nonumber\\
    & = \frac{\sin (\pi t)}{\pi t} \binom{n}{k} \left[ \prod_{j = 1}^k
    \left(1 - \frac{t}{j} \right) \prod_{j = 1}^{n - k} \left(1 +
    \frac{t}{j} \right) \right]^{- 1},  \label{eq:binom:t:prod}
  \end{align}
  where we used $\sin (\pi (t - k)) = (- 1)^k \sin (\pi t)$. If $\alpha$ and
  $\beta$ are integers such that $\alpha \leq 0$ and $\beta > n$, we
  therefore have
  \begin{equation*}
    A^{(s)} (n, t) = \left(\frac{\pi t}{\sin (\pi t)} \right)^s \sum_{k =
     0}^n \binom{n}{k - t}^s = \left(\frac{\pi t}{\sin (\pi t)} \right)^s
     \sum_{k = \alpha}^{\beta - 1} \binom{n}{k - t}^s + O (t^s),
  \end{equation*}
  where the first equality is a consequence of \eqref{eq:binom:t:prod}, while
  the second follows from the added binomial coefficients being $O (t)$ as $t
  \rightarrow 0$. The claim \eqref{eq:A:t:rec} therefore follows if we can
  show that
  \begin{equation*}
    P (n, N) \sum_{k = \alpha}^{\beta - 1} \binom{n}{k - t}^s = O (t^s)
  \end{equation*}
  for large enough $n$.
  
  Since \eqref{eq:franel:ct} after dividing by $\binom{n}{k}^s$ is a rational-function
  identity in $n$ and $k$, the relation \eqref{eq:franel:ct} continues to hold
  if we replace $k$ by $k - t$, for an indeterminate $t$, resulting in
  \begin{equation}
    P (n, N) \binom{n}{k - t}^s = b (n, k + 1 - t) - b (n, k - t) .
    \label{eq:franel:ct:t}
  \end{equation}
  Because $b (n, t) / a (n, t) = R (n, t)$ is a rational function,
  while $a (n, t)$ is an entire function in $t$ for each $n \geq 0$,
  we conclude that, for each large enough $n$ (so that the denominator of $R (n,
  t)$ cannot vanish for all $t$), $b (n, t)$ can have at most finitely many
  poles as a function of $t$. However, for fixed such $n$, the function
  $b (n, t + 1) - b (n, t)$ is entire in $t$ (since the the left-hand side
  in \eqref{eq:franel:ct:t} is a linear combination of entire functions),
  hence $b (n, t)$
  cannot have poles at all. In particular, for large enough $n$, $b (n, t)$ is itself an
  entire function in $t$, and we can then sum \eqref{eq:franel:ct:t} over
  $k$ to obtain
  \begin{equation}
    P (n, N) \sum_{k = \alpha}^{\beta - 1} \binom{n}{k - t}^s = b (n, \beta -
    t) - b (n, \alpha - t) . \label{eq:franel:ct:t:sum}
  \end{equation}
  It therefore remains to show that $b (n, \alpha - t)$ and $b (n, \beta - t)$
  are each $O (t^s)$ as $t\to0$ for some integral $\alpha \leq 0$ and $\beta > n$ of our
  choosing. To that end, fix $n$ and note that, if $\alpha \leq 0$ is an integer, then
  \begin{equation*}
    b (n, \alpha - t) = R (n, \alpha - t) \binom{n}{\alpha - t}^s
  \end{equation*}
  is $O (t^s)$ as $t\to0$, because the binomial coefficient
  \begin{equation*}
  \binom{n}{\alpha - t} = \frac{(-1)^{\alpha+1}}{\binom{n-\alpha}{n}}\,t + O(t^2)
  \end{equation*}
  is $O(t)$, provided that the rational function $r (t) \assign R (n, t)$ (which is
  well-defined for large enough $n$) does not have a pole at $t = \alpha$.
  This is necessarily the case for $\alpha \leq 0$ of large enough absolute value because
  $r (t)$ can have at most finitely many poles. The same argument applies to
  show that $b (n, \beta - t)$ is $O (t^s)$ for large enough integral $\beta > n$.
\end{proof}

\begin{remark}
  \label{rk:A:t:rec:n0}
  The proof above shows that Theorem~\ref{thm:A:t:rec} is true for all $n \in
  \mathbb{Z}_{\geq 0}$ if the denominator of the rational certificate $R
  (n, k)$ has no factor of the form $n - n_0$ for some $n_0 \in
  \mathbb{Z}_{\geq 0}$ (so that $b (n, t)$ is an entire function in $t$ for
  all $n \in \mathbb{Z}_{\geq 0}$).  The computations
  mentioned in Remark~\ref{rk:franel:recs} below show that, for $s \leq
  20$, the minimal recurrence is telescoping and that, up to a
  constant multiple, the denominator of $R (n, k)$ is $(n - k + 1)_m^s$. It is
  natural to expect that these observations continue to be true for all $s$.
\end{remark}

\begin{remark}
  With \eqref{eq:binom:t:prod} in mind, we note that the rational function
  \begin{equation*}
    \frac{n!}{t (t + 1) \cdots (t + n)} = \sum_{k = 0}^n \frac{(- 1)^k
     \binom{n}{k}}{t + k} = \frac{\pi}{\sin (\pi t)} \binom{n}{- t},
  \end{equation*}
  and its powers play a role of building bricks in constructions of
  $\mathbb{Q}$-linear forms in zeta values
  \cite{nesterenko-zeta,zudilin-arithmetic-odd}.
\end{remark}

\section{Proof of the Ap\'ery limits}

This section is devoted to a proof of the following result which, together
with Theorem~\ref{thm:A:t:rec}, establishes the Ap\'ery limits associated to
the generalized Franel numbers $A^{(s)} (n)$ as claimed in
Theorem~\ref{thm:franel:limits}.

\begin{theorem}
  \label{thm:franel:lim:sin}For any $s \in \mathbb{Z}_{> 0}$, we have
  \begin{equation}
    \lim_{n \rightarrow \infty} \frac{A^{(s)} (n, t)}{A^{(s)} (n)} = \left(\frac{\pi t}{\sin (\pi t)} \right)^s, \label{eq:franel:lim:sin}
  \end{equation}
  where the convergence is locally uniform in $t$ \textup(restricted to the unit ball
  $| t | < 1$\textup).
\end{theorem}

That is, we wish to show that
\begin{equation}
  \lim_{n \rightarrow \infty} \frac{\sum_{k = 0}^n \binom{n}{k}^s \left[
  \prod_{j = 1}^k \left(1 - \frac{t}{j} \right) \prod_{j = 1}^{n - k} \left(1 + \frac{t}{j} \right) \right]^{- s}}{\sum_{k = 0}^n \binom{n}{k}^s} =
  \left(\frac{\pi t}{\sin (\pi t)} \right)^s, \label{eq:franel:lim:sin:expl}
\end{equation}
and that the convergence is locally uniform in $t$.

The asymptotics for sums of powers of binomials are known to be, for fixed
$s$,
\begin{equation}
  \sum_{k = 0}^n \binom{n}{k}^s = \frac{2^{n s}}{\sqrt{s (\pi n / 2)^{s - 1}}}
  \left(1 + O \left(\frac{1}{n} \right) \right) . \label{eq:franel:asy}
\end{equation}
For instance, a more precise estimate with additional terms (and which applies
to more general binomial sums) is derived by McIntosh
\cite{mcintosh-binom-asy}. Slightly weaker estimates are derived in
\cite[p.~486--489]{gkp-concrete}, with full details provided in the case $s
= 1$, as well as in \cite{fl-binom-asy}. In each case, the analysis rests on
the fact that the binomial sum is dominated by those terms with $k \approx n /
2$. However, the precise choice of cut-off for the dominant part of the sum
differs between the various approaches. In \cite{mcintosh-binom-asy} the
dominant terms are those corresponding to $k$ satisfying $\left| k -
\frac{n}{2} \right| \leq \varepsilon n$ for suitable $\varepsilon > 0$,
while in \cite{gkp-concrete} this condition is replaced with $\left| k -
\frac{n}{2} \right| \leq \varepsilon n^{1 / 2}$. On the other hand, in
\cite{gkp-concrete}, one restricts to those $k$ in the set
\begin{equation*}
  K_{n, \varepsilon} = \left\{ k \in \mathbb{Z} :
   \left| k - \frac{n}{2} \right| \leq n^{1 / 2 + \varepsilon} \right\} .
\end{equation*}
It is this latter choice that is most suitable for our present purposes.

Naturally, our strategy to establish the limit \eqref{eq:franel:lim:sin:expl}
is to exploit the fact that the sums on the left-hand side are concentrated
around $k \approx n / 2$. For those $k$ and large $n$, we have
\begin{equation*}
  \prod_{j = 1}^k \left(1 - \frac{t}{j} \right) \prod_{j = 1}^{n - k} \left(1 + \frac{t}{j} \right)
  \approx \prod_{j =
   1}^{\infty} \left(1 - \frac{t}{j} \right) \left(1 + \frac{t}{j} \right) =
   \frac{\sin (\pi t)}{\pi t} .
\end{equation*}
On the other hand, this is not true if $k$ is not sufficiently close to $n /
2$; however, we will show that the contribution from these $k$ is overall
negligible. To make this precise, we begin by observing the following desired
behaviour for $k \in K_{n, \varepsilon}$.

\begin{lemma}
  \label{lem:prod2}Fix $\varepsilon \in [0, 1 / 2)$ and $\tau > 0$. Then, for
  all integers $n \geq 0$, all $k \in K_{n, \varepsilon}$ and all $| t |
  \leq \tau$, we have
  \begin{equation*}
    \prod_{j = 1}^k \left(1 - \frac{t}{j} \right) \prod_{j = 1}^{n - k}
     \left(1 + \frac{t}{j} \right) = \frac{\sin (\pi t)}{\pi t} \left(1 + O
     \left(\frac{1}{n^{1 / 2 - \varepsilon}} \right) \right)
  \end{equation*}
  where the implied constant depends on $\varepsilon$ and $\tau$ \textup(but not on
  $t$ or $k$\textup).
\end{lemma}

\begin{proof}
  To begin with, note that
  
  \begin{equation*}
    \prod_{j = 1}^n \left(1 + \frac{t}{j} \right) = \frac{(n + t) !}{n! t!}
     .
  \end{equation*}
  
  In light of the classical
  \begin{equation*}
    \frac{1}{\Gamma (1 + t) \Gamma (1 - t)} = \frac{\sin (\pi t)}{\pi t},
  \end{equation*}
  we therefore need to show that
  \begin{equation}
    \frac{(k - t) !}{k!} \frac{(n - k + t) !}{(n - k) !} = 1 + O \left(\frac{1}{n^{1 / 2 - \varepsilon}} \right) . \label{eq:prod2:1}
  \end{equation}
  To this end, recall Stirling's formula in its logarithmic form, namely,
  \begin{equation}
    \ln (n!) = n \ln (n) - n + \frac{\ln (n)}{2} + \frac{1}{2} \ln (2 \pi) + O
    \left(\frac{1}{n} \right) . \label{eq:stirling:log}
  \end{equation}
  With the assumption that $t = O (1)$, we deduce from
  \eqref{eq:stirling:log} that
  \begin{equation}
    \frac{(n + t) !}{n!} = n^t \left(1 + O \left(\frac{1}{n} \right) \right)
    \label{eq:prod1:nt}
  \end{equation}
  and, therefore,
  \begin{equation*}
    \frac{(k - t) !}{k!} \frac{(n - k + t) !}{(n - k) !} = \left(\frac{n}{k}
     - 1 \right)^t \left(1 + O \left(\frac{1}{k} \right) + O \left(\frac{1}{n - k} \right) \right) .
  \end{equation*}
  The assumption $k \in K_{n, \varepsilon}$ implies that $k = \frac{n}{2} + O
  (n^{1 / 2 + \varepsilon})$ and, in particular,
  \begin{equation*}
    \frac{n}{k} - 1 = \frac{n}{\frac{n}{2} + O (n^{1 / 2 + \varepsilon})} - 1
     = 1 + O \left(\frac{1}{n^{1 / 2 - \varepsilon}} \right),
  \end{equation*}
  leading us to the claimed relation \eqref{eq:prod2:1}.
\end{proof}

On the other hand, for $k \notin K_{n, \varepsilon}$, the products can be
bounded using the following simple observation.

\begin{lemma}
  \label{lem:prod:tail:bound}Fix $\tau \in (0, 1)$. For all integers $n > 0$
  and all $| t | \leq \tau$, we have
  \begin{equation*}
    \left| \prod_{j = 1}^n \left(1 + \frac{t}{j} \right) \right|^{- 1} = O
     (n^{\tau})
  \end{equation*}
  where the implied constant depends on $\tau$ \textup(but not on $t$\textup).
\end{lemma}

\begin{proof}
  Because $| t | \leq \tau < 1$, we have
  \begin{equation*}
    \left| \prod_{j = 1}^n \left(1 + \frac{t}{j} \right) \right|^{- 1}
     \leq \prod_{j = 1}^n \left(1 - \frac{\tau}{j} \right)^{- 1}
  \end{equation*}
  and it can be deduced from \eqref{eq:prod1:nt} that
  \begin{equation*}
    \prod_{j = 1}^n \left(1 - \frac{\tau}{j} \right)^{- 1} = \Gamma (1 -
     \tau) n^{\tau} \left(1 + O \left(\frac{1}{n} \right) \right) .
\qedhere
  \end{equation*}
\end{proof}

In particular, for $0 < k \leq n$ and $| t | \leq \tau < 1$, we
conclude from Lemma~\ref{lem:prod:tail:bound} the crude bound
\begin{equation}
  \left[ \prod_{j = 1}^k \left(1 - \frac{t}{j} \right) \prod_{j = 1}^{n - k}
  \left(1 + \frac{t}{j} \right) \right]^{- s} = O (n^{2 s}),
  \label{eq:prod2:bound}
\end{equation}
which could be easily strenghtened but which suffices for our purposes.

We now write
\begin{equation*}
  a_{s, t} (k, n) = \binom{n}{k}^s \left[ \prod_{j = 1}^k \left(1 -
   \frac{t}{j} \right) \prod_{j = 1}^{n - k} \left(1 + \frac{t}{j} \right)
   \right]^{- s}
\end{equation*}
and follow the approach in \cite[p.~486--489]{gkp-concrete} to prove the
following.

\begin{lemma}
  Fix $s > 0$, $\tau \in (0, 1)$ and $\varepsilon \in (0, 1 / 6)$. Then, for
  all integers $n \geq 0$ and all $k \in K_{n, \varepsilon}$ and all $t$
  such that $| t | \leq \tau$, we have
  \begin{equation}
    \sum_{k = 0}^n a_{s, t} (k, n) = \frac{2^{n s}}{\sqrt{s (\pi n / 2)^{s -
    1}}} \left[ \frac{\pi t}{\sin (\pi t)} \right]^s \left(1 + O \left(\frac{1}{n^{1 / 2 - 3 \varepsilon}} \right) \right) \label{eq:a:sum}
  \end{equation}
  where the implied constant depends on $s$, $\varepsilon$ and $\tau$ \textup(but not
  on $t$\textup).
\end{lemma}

\begin{proof}
  Proceeding as in \cite{gkp-concrete}, we obtain that, for $k \in K_{n,
  \varepsilon}$,
  \begin{equation}
    \binom{n}{k}^s = \left[ \frac{2^{n + 1}}{\sqrt{2 \pi n}} e^{- 2 (k - n /
    2)^2 / n} \right]^s \left(1 + O \left(\frac{1}{n^{1 / 2 - 3
    \varepsilon}} \right) \right) . \label{eq:binom:main:asy}
  \end{equation}
  Combined with Lemma~\ref{lem:prod2}, we find
  \begin{equation*}
    a_{s, t} (k, n) = \left[ \frac{2^{n + 1}}{\sqrt{2 \pi n}} e^{- 2 (k - n /
     2)^2 / n}  \frac{\pi t}{\sin (\pi t)} \right]^s \left(1 + O \left(\frac{1}{n^{1 / 2 - 3 \varepsilon}} \right) \right) .
  \end{equation*}
  In other words, we have, for $k \in K_{n, \varepsilon}$,
  \begin{equation*}
    a_{s, t} (k, n) = b_{s, t} (k, n) + O (c_{s, t} (k, n)),
  \end{equation*}
  where
  \begin{equation*}
    b_{s, t} (k, n) = \left[ \frac{2^{n + 1}}{\sqrt{2 \pi n}} e^{- 2 (k - n /
     2)^2 / n}  \frac{\pi t}{\sin (\pi t)} \right]^s
  \end{equation*}
  and (because $\pi t / \sin (\pi t) = O (1)$)
  \begin{equation*}
    c_{s, t} (k, n) = 2^{n s} e^{- 2 s (k - n / 2)^2 / n}  \frac{1}{n^{(s +
     1) / 2 - 3 \varepsilon}} .
  \end{equation*}
  We now apply the tail-exchange method and estimate
  \begin{align}
    \sum_k a_{s, t} (k, n) & = \sum_k b_{s, t} (k, n) + O \left(\sum_{k
    \notin K_{n, \varepsilon}} | a_{s, t} (k, n) | \right) + O \left(\sum_{k
    \notin K_{n, \varepsilon}} | b_{s, t} (k, n) | \right) \nonumber\\
    &\quad + O \left(\sum_{k \in K_{n, \varepsilon}} | c_{s, t} (k, n) |
    \right) .  \label{eq:abc:sum}
  \end{align}
  The idea, of course, being that the first term on the right-hand side of
  \eqref{eq:abc:sum}, namely
  \begin{align}
    \sum_k b_{s, t} (k, n) & = \left[ \frac{2^{n + 1}}{\sqrt{2 \pi n}}
    \,\frac{\pi t}{\sin (\pi t)} \right]^s \sum_k e^{- 2 s (k - n / 2)^2 / n}
    \nonumber\\
    & = \left[ \frac{2^{n + 1}}{\sqrt{2 \pi n}} \,\frac{\pi t}{\sin (\pi t)}
    \right]^s \sqrt{\frac{\pi n}{2 s}} \left(1 + O \left(e^{- \frac{n
    \pi^2}{2 s}} \right) \right) \nonumber\\
    & = \frac{2^{n s}}{\sqrt{s (\pi n / 2)^{s - 1}}} \left[ \frac{\pi
    t}{\sin (\pi t)} \right]^s \left(1 + O \left(e^{- \frac{n \pi^2}{2 s}}
    \right) \right),  \label{eq:b:sum}
  \end{align}
  provides the asymptotics for the left-hand side of \eqref{eq:abc:sum} while
  the other terms are negligible in comparison. Indeed,
  \begin{equation*}
    \sum_{k \in K_{n, \varepsilon}} | c_{s, t} (k, n) | \leq \sum_k
     c_{s, t} (k, n) = \frac{2^{n s}}{n^{(s + 1) / 2 - 3 \varepsilon}} \sum_k
     e^{- 2 s (k - n / 2)^2 / n}
  \end{equation*}
  is asymptotically smaller than \eqref{eq:b:sum} provided that $3 \varepsilon
  < \frac{1}{2}$ (note that adding this contribution to \eqref{eq:b:sum}
  requires adjusting the error term in \eqref{eq:b:sum} to the one claimed in
  \eqref{eq:a:sum}). Likewise,
  \begin{align*}
    \sum_{k \notin K_{n, \varepsilon}} | b_{s, t} (k, n) | & = \left|
    \frac{2^{n + 1}}{\sqrt{2 \pi n}} \, \frac{\pi t}{\sin (\pi t)} \right|^s
    \sum_{k \notin K_{n, \varepsilon}} e^{- 2 s (k - n / 2)^2 / n}
  \end{align*}
  is asymptotically smaller than \eqref{eq:b:sum} because the right-hand side
  sum is $O (n^{- M})$ for all $M$ (here, we use that $\varepsilon > 0$).
  Thirdly, by \eqref{eq:prod2:bound},
  \begin{equation*}
    \sum_{k \notin K_{n, \varepsilon}} | a_{s, t} (k, n) | = O (n^{2 s})
     \sum_{k \notin K_{n, \varepsilon}} \binom{n}{k}^s
  \end{equation*}
  and the right-hand side sum is bounded by $n$ times its largest term, which
  is bounded by the one corresponding to $k^{\ast} = \left\lfloor \frac{n}{2}
  + n^{1 / 2 + \varepsilon} \right\rfloor \in K_{n, \varepsilon}$. In
  particular, applying \eqref{eq:binom:main:asy} to that term, reveals that
  \begin{equation*}
    \sum_{k \notin K_{n, \varepsilon}} | a_{s, t} (k, n) | =
     \binom{n}{k^{\ast}}^s O (n^{2 s + 1}) = \left[ \frac{2^{n + 1}}{\sqrt{2
     \pi n}} e^{- 2 (k^{\ast} - n / 2)^2 / n} \right]^s O (n^{2 s + 1})
  \end{equation*}
  is asymptotically smaller than \eqref{eq:b:sum} as well.
\end{proof}

Combining \eqref{eq:franel:asy} and \eqref{eq:a:sum}, we conclude the desired
limit \eqref{eq:franel:lim:sin:expl}, including the required uniform
convergence.

\section{Lower bounds for telescoping recurrences}\label{sec:ct:order}

We are now in a position to apply the results on Ap\'ery limits to prove
Theorem~\ref{thm:franel:conj:ct}. That is, we wish to conclude that any
telescoping recurrence satisfied by $A^{(s)} (n)$ has order at least $\lfloor
(s + 1) / 2 \rfloor$.

\begin{proof}[Proof of Theorem~\textup{\ref{thm:franel:conj:ct}}]
  By Theorem~\ref{thm:franel:limits}, any telescoping recurrence satisfied by
  $A^{(s)} (n)$ is also solved, for large enough $n$, by the $\lfloor (s + 1)
  / 2 \rfloor$ sequences $A^{(s)}_j (n) \in\mathbb{Q}$ defined in \eqref{eq:A:j},
  where $j \in \{ 0, 1, \ldots, \lfloor (s - 1) / 2 \rfloor \}$. We recall
  from \cite[Theorem~8.2.1]{aeqb} that a recurrence with polynomial
  coefficients has order $r$ if and only if the space of its solutions, upon
  identifying sequences that eventually agree, has dimension $r$.
  
  Therefore, to conclude that any telescoping recurrence satisfied by $A^{(s)}
  (n)$ has order at least $r = \lfloor (s + 1) / 2 \rfloor$, it suffices to
  show that the $r$ solutions $A^{(s)}_j (n)$, $j \in \{ 0, 1, \ldots, r - 1
  \}$, upon this identification, are linearly independent.
  As these solutions are rational-valued, assuming their linear dependence,
  there must necessarily exist a dependence relation over $\mathbb Q$.  This means that
  \begin{equation}
    0 = \sum_{j = 0}^{r - 1} \lambda_j A_j^{(s)} (n), \quad \lambda_j \in \mathbb{Q},
  \label{eq:A:lindep}
  \end{equation}
  for large enough $n$. Now, upon
  dividing \eqref{eq:A:lindep} by $A^{(s)} (n)$ and taking the limit as $n \rightarrow \infty$, we find
  out that
  \begin{equation*}
    0 = \lim_{n \rightarrow \infty} \sum_{j = 0}^{r - 1} \lambda_j
     \frac{A^{(s)}_j (n)}{A^{(s)} (n)} = \sum_{j = 0}^{r - 1} \lambda_j
     \varphi_j \pi^{2 j},
  \end{equation*}
  where the latter equality uses the Ap\'ery limits established in
  Theorem~\ref{thm:franel:limits}. Since $\lambda_j \varphi_j \in \mathbb{Q}$,
  the transcendence of $\pi$ implies that $\lambda_j \varphi_j = 0$ for all $j
  \in \{ 0, 1, \ldots, r - 1 \}$. We know that $\varphi_j \neq 0$, so we must
  have $\lambda_j = 0$ for all $j \in \{ 0, 1, \ldots, r - 1 \}$, proving the
  desired linear independence of the $r$ solutions $A^{(s)}_j (n)$.
\end{proof}

\begin{remark}
  \label{rk:franel:recs}The computations of Perlstadt
  \cite{perlstadt-franel} and McIntosh \cite{mcintosh-phd} show that a
  telescoping recurrence equation of (the conjectured to be minimal) order $m
  = \lfloor (s + 1) / 2 \rfloor$ exists for $s \leq 10$. We have extended
  these computations to all $s \leq 20$ using Koutschan's implementation
  \cite{koutschan-phd} \texttt{HolonomicFunctions} in Mathematica and, in
  each case, obtained a recurrence of order $m$ (the minimality of these
  recurrence operators was then confirmed using the function
  \texttt{MinimalRecurrence} from the \texttt{LREtools} Maple package).
  
  These computations suggest that a minimal-order recurrence of $A^{(s)} (n)$
  can always be obtained via creative telescoping. Moreover, Alin Bostan
  observes that this minimal recurrence of order
  $m$ has polynomial coefficients of degree
  \begin{equation*}
d = \begin{cases}
\frac13 m (m^2 - 1) + 1, & \text{for even $s$}, \\[1mm]
\frac13m^3-\frac12m^2+\frac23m + \frac{(-1)^m-1}4, & \text{for odd $s$}.
\end{cases}
  \end{equation*}
In particular, the degree appears to grow like $s^3 / 24$ (rather than being
  bounded by $s - 1$ as Franel incorrectly predicted in \cite{franel95}). As for the rational
  certificate, when written in lowest terms and with integer coefficients, we further find out that its
  denominator is given by
  \begin{equation*}
    (n - k + 1)_m^s = \prod_{j = 1}^m (n - k + j)^s
  \end{equation*}
  and, thus, has degree $m s$ in each of $n$ and $k$. The corresponding
  numerator has degree $m s + \delta_2 (s)$ in the variable $k$ (as used in
  Remark~\ref{rk:A:t:rec:n0}), where the delta notation is for $\delta_r (s) = 0$ unless $r$ divides $s$
  in which case $\delta_r (s) = 1$. At the same time, its degree in $n$ is $d
  + s (s - 1 - \delta_2 (s)) / 2 - \delta_6 (s)$. Moreover, the numerator is
  \begin{equation*}
    k^s \prod_{j \geq 1} (n + j)^{\max (0, s + 2 - 4 j - (- 1)^s)}
  \end{equation*}
  times a (large) irreducible factor. These observations hold true for $s
  \leq 20$, and it is natural to expect that the patterns persist for
  larger $s$ as well.
  
  If desired, the above computations can readily be extended to larger $s$.
  Readers interested in computing telescoping recurrence equations for large
  $s$ might find value in considering a guess-and-prove approach (with the
  above observations taken into account) such as described, for instance, in
  \cite{pillwein-pos} for a different hypergeometric sum.
\end{remark}

\section{Conclusions}

We have explicitly determined the Ap\'ery limits associated to the
generalized Franel numbers, resolving the explicit conjectures in
\cite{cs-aperylimits}. As a novel application of Ap\'ery limits, we proved
in Theorem~\ref{thm:franel:conj:ct} that Franel's conjecture is true if the
minimal-order recurrence satisfied by $A^{(s)} (n)$ is a telescoping
recurrence equation. It would be useful to establish general conditions under
which it can be guaranteed that creative telescoping is able to determine a
recurrence of minimal order. As a rare result of this type, we mention that
Schneider \cite[Corollary~7.4]{schneider-indep} proves that creative
telescoping finds a minimal (inhomogeneous) recurrence for certain sums over
hypergeometric terms $a (n, k)$ where the summation bounds are independent of
$n$ but finite. On the other hand, we refer to Paule
\cite[Section~11.2]{paule-contiguous} for an example in which creative
telescoping is not able to find a recurrence of minimal order. We echo
Chyzak's \cite[p.~52]{chyzak-abc} comment that ``a theoretical explanation
is still missing and would be welcome in order to design algorithms for
minimal-order annihilators.'' From a different point of view, it is not
necessarily that creative telescoping suffers from missing a minimal-order
recurrence for a given $D$-finite sequence $A (n)$ but that the sequence
itself always possesses multiple hypergeometric representations, also as
multiple binomial sums, and that we \emph{a priori} have no knowledge on
which of those the algorithm will produce the optimal outcome.

As mentioned in the introduction, Stoll \cite{stoll-rec-bounds}, as well as
Cusick \cite{cusick-rec}, construct recurrences for the generalized Franel
numbers (of the conjectured order). It would be of interest to see if these
constructions can be augmented to show that they actually result in
telescoping recurrences.

As noted in the introduction (see also Remark~\ref{rk:mcintosh} there), there
is a shortage of general results that make it possible to prove lower bounds
on the order of recurrences satisfied by $D$-finite sequences. We expect that
the present approach can be applied to other families of binomial sums to
compute the corresponding Ap\'ery limits and to prove lower bounds for their
minimal telescoping recurrences.

\subsection*{Acknowledgements}

We thank Alin Bostan and Christoph Koutschan for their generous and expertly
advice on the current state of computer algebra with regards to computing (for
fixed parameter $s$) recurrences for the generalized Franel numbers and
certifying their minimality. In particular, Alin kindly shared the empirical
observation on the degrees of the minimal-order recurrences that is reported
in Remark~\ref{rk:franel:recs}.

The first author gratefully acknowledges support through a Collaboration Grant
(\#514645) awarded by the Simons Foundation.


\end{document}